\documentclass[reqno]{amsart}
\usepackage[a4paper,
bindingoffset=0.2in,
left=1in,
right=1in,
top=1in,
bottom=1in,
footskip=.25in]{geometry}
\usepackage[english]{babel}
\usepackage{amsthm}
\usepackage{mathrsfs}
\usepackage{amsmath}
\usepackage{amsfonts}
\usepackage{mathtools}
\usepackage{xcolor}
\usepackage{wrapfig}
\usepackage{hyperref}
\usepackage{float}
\usepackage{pgfplots}
\usepackage{amssymb,enumitem}
\pgfplotsset{compat=1.18,width=10cm}
\usepackage{subcaption}
\usepackage{tikz}
\usepackage{amssymb} 
\usepackage[numbers]{natbib}
\newtheorem{theorem}{Theorem}
\newtheorem{lemma}[theorem]{Lemma}

\newtheorem{remark}[theorem]{Remark}

\theoremstyle{definition}

\theoremstyle{remark}
\numberwithin{theorem}{section}
\DeclareUnicodeCharacter{2217}{*}
\DeclareUnicodeCharacter{211D}{\mathbb{R}}

\newcommand{\be}{\begin{equation}}
\newcommand{\ee}{\end{equation}}
\numberwithin{equation}{section}
\allowdisplaybreaks
\author[Komal Verma]
{Komal Verma}
\address{Komal Verma\hfill\break
Department of Mathematics\newline
Birla Institute of Technology and Science Pilani \newline
Pilani Campus, Vidya Vihar \newline
Pilani, Jhunjhunu \newline
Rajasthan, India - 333031}
\email{p20230058@pilani.bits-pilani.ac.in;kv8802727566@gmail.com }

\author[Gaurav Dwivedi ]
{Gaurav Dwivedi}
\address{Gaurav Dwivedi \hfill\break
Department of Mathematics\newline
Birla Institute of Technology and Science Pilani \newline
Pilani Campus, Vidya Vihar \newline
Pilani, Jhunjhunu \newline
Rajasthan, India - 333031}
\email{gaurav.dwivedi@pilani.bits-pilani.ac.in}
\begin{document}
\title[ Semipositone Problem ] {Positive Solutions for a Mixed Local-Nonlocal Problem with Semipositone Nonlinearity }

\begin{abstract}
In this article, we prove the existence of at least one positive solution for the mixed local-nonlocal semipositone problem 
\begin{equation*}
    \left\{
    \begin{aligned}
        -\Delta_p u+ (-\Delta)^s_p u &= \lambda f(u) && \text{in } \Omega, \\
        u &= 0 && \text{in } \mathbb{R}^N \setminus \Omega,
    \end{aligned}
    \right.
\end{equation*}
 using mountain pass arguments, comparison principles and regularity principles.
\end{abstract}
\maketitle
\section{Introduction}
We study the existence of positive solutions to
\begin{equation}\label{eq:1.1}
\left\{
\begin{aligned}
-\Delta_p u+(-\Delta)^s_p u &= \lambda f(u) && \text{in } \Omega,\\
u &=0 && \text{in } \mathbb{R}^N\setminus\Omega,
\end{aligned}
\right.
\end{equation}
where $N>2$, $\Omega\subset\mathbb{R}^N$ is a bounded domain with smooth boundary, $s\in(0,1)$, $1<p$, $sp<N$, and $\lambda>0$. The function $f:\mathbb{R}\to\mathbb{R}$ is continuous. Here
\[
\Delta_p u=\operatorname{div}(|\nabla u|^{p-2}\nabla u)
\]
is the $p$-Laplacian, and $(-\Delta)^s_p$ denotes the fractional $p$-Laplacian
\[
(-\Delta)^s_p u(x)=2\lim_{\varepsilon\to0^+}
\int_{\{|x-y|>\varepsilon\}}
\frac{|u(x)-u(y)|^{p-2}(u(x)-u(y))}{|x-y|^{N+sp}}\,dy .
\]

Throughout the paper we assume the following hypotheses.

\begin{enumerate}[label=(H\arabic*)]

\item There exist $q\in\left(p-1,\min\left\{\frac{sp}{N}p_s^*,\,p_s^*-1\right\}\right)$ and $A,B>0$ such that
\begin{equation}\label{eq:1.2}
\begin{cases}
A(s^q-1)\le f(s)\le B(s^q+1) & \text{for } s>0,\\
f(s)=0 & \text{for } s\le -1,
\end{cases}
\end{equation}
where $p_s^*=\dfrac{Np}{N-sp}$ is the fractional critical Sobolev exponent.

\item Let $F(t)=\displaystyle\int_0^t f(s)\,ds$. There exist $A_1,B_1,C_1>0$ such that
\begin{equation}\label{eq:1.3}
F(s)\le B_1(|s|^{q+1}+1)\qquad \text{for all } s\in\mathbb{R},
\end{equation}
and
\begin{equation}\label{eq:1.4}
A_1(s^{q+1}-C_1)\le F(s)\qquad \text{for all } s\ge0 .
\end{equation}

\item (Ambrosetti--Rabinowitz condition) There exist $\theta>p$ and $M\in\mathbb{R}$ such that
\begin{equation}\label{eq:1.5}
sf(s)\ge \theta F(s)+M
\qquad \text{for all } s\in\mathbb{R}.
\end{equation}

\end{enumerate}

The study of positive solutions to Dirichlet problems with semipositone nonlinearities has received considerable attention. Castro and Shivaji~\cite{castro1988non} were among the first to investigate nonnegative solutions for a class of non-positone problems in one dimension. This was followed by Castro and Shivaji~\cite{castro1989nonnegative}, who established the existence of nonnegative radial solutions for a non-positone problem involving the Laplacian in the unit ball. Subsequently, Ali et al.~\cite{ali1993uniqueness} examined the uniqueness and stability of nonnegative solutions for semipositone problems on a ball. Since then, numerous works have addressed positive solutions of Dirichlet problems with semipositone nonlinearities; see, for instance, \cite{biswas2024study,brown1982simple,caldwell2007positive,castro1994uniqueness,castro1998positive} and the references therein.

A variety of analytical techniques have been employed to study the existence and nonexistence of solutions, including sub- and supersolution methods, degree theory, fixed point theory, and bifurcation techniques; see \cite{afrouzi2008critical,agarwal2006existence,ambrosetti1994positive,biswas2025semipositone,dhanya2025positive}. In the present setting, the AR-type condition~\eqref{eq:1.5} enables the use of the mountain pass theorem to establish the existence of a solution.

More recently, motivated by the work of Biagi et al.~\cite{biagi2022mixed}, mixed local--nonlocal problems have been widely investigated; see \cite{anthal2025mixed,biagi2024brezis,garain2023class,su2022regularity}. In this direction, De Filippis and Mingione~\cite{de2024gradient} obtained important regularity results for equations involving both local and nonlocal operators. In particular, under the assumption $p>sq$, they proved interior Hölder continuity of the gradients and global almost-Lipschitz regularity. Furthermore, Antonini et al.~\cite{antonini2025global} established that weak solutions to the Dirichlet problem are $C^{1,\theta}$ regular up to the boundary and derived comparison and maximum principles, which will play a key role in our analysis.

Our results extend \citep[Theorem~1.1]{castro2016existence} and \citep[Theorem~1.1]{lopera2023existence}. Castro et al.~\cite{castro2016existence} proved the existence of positive solutions for a semipositone $p$-Laplacian problem, while Lopera et al.~\cite{lopera2023existence} obtained analogous results for the fractional $p$-Laplacian using mountain pass arguments combined with comparison principles, regularity results, and a priori estimates.

In this paper, we show that the associated energy functional $E_{\lambda}$ admits a mountain pass type critical point for sufficiently small $\lambda>0$. By exploiting regularity properties of the operator $-\Delta_p u+(-\Delta)^s_p u$, together with comparison and maximum principles, we further prove that for the same range of $\lambda>0$, problem~\eqref{eq:1.1} admits a weak positive solution.

The main result of this paper is as follows:

\begin{theorem}\label{main}
Let $\Omega \subset \mathbb{R}^N$ be a bounded open set with smooth boundary $\partial\Omega$. Assume that hypotheses $(H1)-(H3)$ hold. Then there exists $\lambda^*>0$ such that for every $\lambda \in (0,\lambda^*)$, problem \eqref{eq:1.1} admits a weak positive solution 
$u_\lambda \in C^{1,\theta}(\overline{\Omega})$
for some $\theta \in (0,1)$.
\end{theorem}

\section{Preliminaries}
Let $s\in(0,1)$, $1\le p<\infty$, and let $\Omega\subset\mathbb{R}^N$ be a bounded open set with a smooth boundary. Define
\[
\mathcal{X}_{1,p}(\Omega):=\{u\in W^{1,p}(\mathbb{R}^N): u=0 \text{ a.e. in } \mathbb{R}^N\setminus\Omega\},
\]
endowed with the norm
\[
\rho(u):=\Big(\|\nabla u\|_p^p+[u]_{s,\mathbb{R}^N}^p\Big)^{1/p},
\]
where
\[
[u]_{s,\mathbb{R}^N}^p:=\iint_{\mathbb{R}^N\times\mathbb{R}^N}\frac{|u(x)-u(y)|^p}{|x-y|^{N+sp}}\,dx\,dy
\]
is the Gagliardo seminorm. Here $\|\cdot\|_q$ denotes the norm in $L^q(\Omega)$ for $1\le q\le\infty$.

Since $\Omega$ is smooth, we have $\mathcal{X}_{1,p}(\Omega)=W_0^{1,p}(\Omega)$ (see \citep[Proposition~9.18]{brezis2011functional}). Moreover, $C_0^\infty(\Omega)$ is dense in $\mathcal{X}_{1,p}(\Omega)$, and $\mathcal{X}_{1,p}(\Omega)$ is separable and reflexive.

The Sobolev embedding theorem, i.e., $\mathcal{X}_{1,p}(\Omega) \hookrightarrow L^q(\Omega)$, provides a constant $S_1 >0$ such that 
\begin{equation}\label{eq:2.6}
\|v\|_{q+1}\le S_1\,\rho(v)
\qquad \text{for all } v\in\mathcal{X}_{1,p}(\Omega).
\end{equation}

For $s\in\mathbb{R}$, set $\omega_p(s)=|s|^{p-2}s$. A weak solution of \eqref{eq:1.1} is a function $u\in\mathcal{X}_{1,p}(\Omega)$ such that for all $\varphi\in C_0^\infty(\Omega)$,
\begin{equation}
\int_{\Omega} |\nabla u|^{p-2}\nabla u\cdot\nabla\varphi\,dx
+\iint_{\mathbb{R}^N\times\mathbb{R}^N}
\omega_p(u(x)-u(y))(\varphi(x)-\varphi(y))
\frac{dx\,dy}{|x-y|^{N+sp}}
=
\lambda\int_{\Omega} f(u)\varphi\,dx .
\end{equation}

Define the functional $E_\lambda:\mathcal{X}_{1,p}(\Omega)\to\mathbb{R}$ by
\begin{equation}\label{eq2.3}
E_\lambda(u)
=\frac{1}{p}\int_{\Omega}|\nabla u|^p\,dx
+\frac{1}{p}\iint_{\mathbb{R}^N\times\mathbb{R}^N}
\frac{|u(x)-u(y)|^p}{|x-y|^{N+sp}}\,dx\,dy
-\lambda\int_{\Omega}F(u)\,dx
=\frac{1}{p}\rho(u)^p-\lambda\int_{\Omega}F(u)\,dx .
\end{equation}

It is well known that $E_\lambda\in C^1(\mathcal{X}_{1,p}(\Omega),\mathbb{R})$ and
\[
\begin{aligned}
\langle E_\lambda'(u),\varphi\rangle
=&\int_{\Omega} |\nabla u|^{p-2}\nabla u\cdot\nabla\varphi\,dx
+\iint_{\mathbb{R}^N\times\mathbb{R}^N}
\omega_p(u(x)-u(y))(\varphi(x)-\varphi(y))
\frac{dx\,dy}{|x-y|^{N+sp}}\\
&-\lambda\int_{\Omega}f(u)\varphi\,dx .
\end{aligned}
\]
Hence, the critical points of $E_\lambda$ coincide with the weak solutions of \eqref{eq:1.1}.

Finally, we set 
\begin{equation}\label{r}
r=\frac{1}{q+1-p}>0,
\end{equation}
which will be used throughout the paper.
\begin{lemma}\label{lemma:2.1}
Let $\varphi \in C_0^{\infty}(\Omega)$ be a positive function with $\rho(\varphi)=1$. 
Then there exists $\lambda_1>0$ such that for all $\lambda\in(0,\lambda_1)$,
\[
E_\lambda(m)\le 0,
\]
where $m=c\lambda^{-r}\varphi$, $r$ is defined in \eqref{r}, and
$c=\Big(2p^{-1}A_1^{-1}\|\varphi\|_{q+1}^{-(q+1)}\Big)^r .$
\end{lemma}

\begin{proof}
Let $m=\ell\varphi$ with $\ell=c\lambda^{-r}$. Using \eqref{eq:1.4} and $\rho(\varphi)=1$, we obtain
\[
E_\lambda(m)
=\frac{1}{p}\rho(\ell\varphi)^p-\lambda\int_\Omega F(\ell\varphi)\,dx
\le \frac{\ell^p}{p}-\lambda A_1\ell^{q+1}\|\varphi\|_{q+1}^{q+1}
+\lambda A_1C_1|\Omega|.
\]
Since $r=\frac{1}{q+1-p}$, we have $rp=r(q+1)-1$. Substituting $\ell=c\lambda^{-r}$ gives
\[
E_\lambda(m)
\le \lambda^{-rp}\!\left(\frac{c^p}{p}-A_1c^{q+1}\|\varphi\|_{q+1}^{q+1}\right)
+\lambda A_1C_1|\Omega|.
\]
By the choice of $c$,
\[
A_1c^{q+1}\|\varphi\|_{q+1}^{q+1}=\frac{2c^p}{p},
\]
and hence
\[
E_\lambda(m)\le -\frac{c^p}{p}\lambda^{-rp}+\lambda A_1C_1|\Omega|.
\]
Therefore $E_\lambda(m)\le0$ whenever
\[
\lambda A_1C_1|\Omega|\le \frac{c^p}{p}\lambda^{-rp},
\]
that is,
\[
\lambda\le
\left(\frac{c^p}{pA_1C_1|\Omega|}\right)^{\frac{1}{rp+1}}
=: \lambda_1 .
\]
Hence, for all $\lambda\in(0,\lambda_1),$  $E_\lambda(m)\le0$.
\end{proof}

\begin{lemma}\label{lemma:2.2}
Let 
\[
\tau=\min\{(2pS_1^{q+1}B_1)^{-r},\,c\}>0,
\]
where $S_1>0$ is the constant in \eqref{eq:2.6}. Then there exist constants $c_1>0$ and $0<\lambda_2<1$ such that if $\rho(u)=\tau\lambda^{-r}$, then
\[
E_\lambda(u)\ge c_1(\tau\lambda^{-r})^p
\qquad \text{for all } \lambda\in(0,\lambda_2).
\]
\end{lemma}
\begin{proof}
Let $u\in\mathcal{X}_{1,p}(\Omega)$ with $\rho(u)=\tau\lambda^{-r}$. Using \eqref{eq:1.3} and the Sobolev embedding \eqref{eq:2.6}, we obtain
\[
\begin{aligned}
E_\lambda(u)
&=\frac{1}{p}\rho(u)^p-\lambda\int_\Omega F(u)\,dx \\
&\ge \frac{1}{p}\rho(u)^p-\lambda B_1\|u\|_{q+1}^{q+1}-\lambda B_1|\Omega|\\
&\ge \frac{1}{p}\tau^p\lambda^{-rp}
-\lambda B_1S_1^{q+1}(\tau\lambda^{-r})^{q+1}
-\lambda B_1|\Omega|.
\end{aligned}
\]
Since $r=\frac{1}{q+1-p}$, we have $r(q+1)=rp+1$. Hence
\[
E_\lambda(u)\ge 
\lambda^{-rp}\!\left(\frac{\tau^p}{p}-B_1S_1^{q+1}\tau^{q+1}\right)
-\lambda B_1|\Omega|.
\]
By the definition of $\tau$, 
\[
B_1S_1^{q+1}\tau^{q+1}\le \frac{\tau^p}{2p},
\]
and therefore
\[
E_\lambda(u)\ge 
\frac{\tau^p}{2p}\lambda^{-rp}-\lambda B_1|\Omega|.
\]
Choose $c_1=\frac{1}{4p}$. Then
\[
E_\lambda(u)\ge c_1(\tau\lambda^{-r})^p
\]
provided
\[
\lambda^{1+rp}\le \frac{\tau^p}{4pB_1|\Omega|}.
\]
Thus, the result holds for
\[
\lambda_2=\left(\frac{\tau^p}{4pB_1|\Omega|}\right)^{\frac{1}{1+rp}}.
\]
\end{proof}

To apply the Mountain Pass Theorem, it remains to verify that $E_\lambda$ satisfies the Palais--Smale condition, which will be established in the next lemma.
\begin{lemma}\label{lemma:2.3}
There exists a constant $c_2>0$ such that for all 
$\lambda\in (0,\lambda_3)$, where $\lambda_3=\min\{\lambda_1,\lambda_2\}$,
the functional $E_\lambda$ possesses a critical point $u_\lambda$ satisfying
\[
c_1\lambda^{-rp}\le E_\lambda(u_\lambda)\le c_2\lambda^{-rp},
\]
where $c_1>0$ is the constant given in Lemma~\ref{lemma:2.2}.
\end{lemma}

\begin{proof}
We first show that $E_\lambda$ satisfies the Palais--Smale condition. 
Let $\{u_n\}\subset \mathcal{X}_{1,p}(\Omega)$ be a sequence such that
\[
E_\lambda(u_n)\ \text{is bounded}, \qquad 
E_\lambda'(u_n)\to 0 \quad \text{in } \mathcal{X}_{1,p}(\Omega)^* \text{ when }n \rightarrow \infty.
\]

Using the Ambrosetti--Rabinowitz condition \eqref{eq:1.5}, we obtain
\[
E_\lambda(u_n)-\frac{1}{\theta}\langle E_\lambda'(u_n),u_n\rangle
=
\left(\frac{1}{p}-\frac{1}{\theta}\right)\rho(u_n)^p
+\frac{\lambda M}{\theta}|\Omega|.
\]
Since $\theta>p$ and $\{E_\lambda(u_n)\}$ is bounded while
$E_\lambda'(u_n)\to0$, it follows that $\{u_n\}$ is bounded in
$\mathcal{X}_{1,p}(\Omega)$.

Hence, up to a subsequence, $u_n\rightharpoonup u$ weakly in
$\mathcal{X}_{1,p}(\Omega)$ and
\[
u_n\to u \quad \text{strongly in } L^{q+1}(\Omega).
\]
Using the monotonicity of the operators
$|\nabla u|^{p-2}\nabla u$ and $\omega_p(u(x)-u(y))$, together with
$E_\lambda'(u_n)\to0$, one obtains
\[
\rho(u_n)\to \rho(u),
\]
which implies $u_n\to u$ strongly in $\mathcal{X}_{1,p}(\Omega)$.
Therefore $E_\lambda$ satisfies the Palais--Smale condition.

Next, from Lemma~\ref{lemma:2.1}, for $0\le \ell\le c\lambda^{-r}$ we have
\[
E_\lambda(\ell\varphi)
\le \frac{\ell^p}{p}+\lambda A_1C_1|\Omega|
\le \left(\frac{c^p}{p}+A_1C_1|\Omega|\right)\lambda^{-rp}
=:c_2\lambda^{-rp}.
\]
Hence
\[
\max_{0\le \ell\le c\lambda^{-r}}E_\lambda(\ell\varphi)
\le c_2\lambda^{-rp}.
\]

By Lemmas~\ref{lemma:2.1} and \ref{lemma:2.2}, the functional $E_\lambda$
has the mountain pass geometry. Since it also satisfies the
Palais--Smale condition, the Mountain Pass Theorem yields a critical
point $u_\lambda\in\mathcal{X}_{1,p}(\Omega)$ such that
\[
E_\lambda(u_\lambda)
=\inf_{\gamma\in\Gamma}\max_{t\in[0,1]}E_\lambda(\gamma(t)),
\]
where
\[
\Gamma=\{\gamma\in C([0,1],\mathcal{X}_{1,p}(\Omega)):\,
\gamma(0)=0,\ \gamma(1)=c\lambda^{-r}\varphi\}.
\]

Finally, Lemma~\ref{lemma:2.2} and the above estimate imply
\[
c_1\lambda^{-rp}\le E_\lambda(u_\lambda)\le c_2\lambda^{-rp}.
\]
\end{proof}
\begin{remark}
\,
\begin{enumerate}
    \item[(i)] Since $E'_\lambda(u_\lambda)=0$, we have
\[
\rho(u_\lambda)^p
=\lambda\int_\Omega f(u_\lambda)u_\lambda\,dx .
\]
Using the Ambrosetti--Rabinowitz condition \eqref{eq:1.5}, we obtain
\[
\begin{aligned}
\left(\frac{1}{p}-\frac{1}{\theta}\right)\rho(u_\lambda)^p
&=\frac{1}{p}\rho(u_\lambda)^p
-\frac{\lambda}{\theta}\int_\Omega f(u_\lambda)u_\lambda\,dx \\
&\le
\frac{1}{p}\rho(u_\lambda)^p
-\frac{\lambda}{\theta}\int_\Omega f(u_\lambda)u_\lambda\,dx
+\frac{\lambda}{\theta}M|\Omega| \\
&\le
\frac{1}{p}\rho(u_\lambda)^p
-\lambda\int_\Omega F(u_\lambda)\,dx
=E_\lambda(u_\lambda).
\end{aligned}
\]
By Lemma~\ref{lemma:2.3},
\[
E_\lambda(u_\lambda)\le c_2\lambda^{-rp}.
\]
Hence there exists a constant $c_3>0$, independent of $\lambda$, such that
\[
\rho(u_\lambda)^p\le c_3\,\lambda^{-rp}.
\]
Consequently,
\begin{equation}\label{eq:2.18}
\rho(u_\lambda)\le c_3\,\lambda^{-r}.
\end{equation}
    \item[(ii)] The weak solution $u_{\lambda}$ obtained in Lemma~\ref{lemma:2.3} is in $L^{\infty}(\Omega)$ by the similar argument used in the proof of \citep[Theorem~4.1]{biagi2024brezis}.
\end{enumerate}
\end{remark}

\begin{lemma}\label{lemma:2.5}
Let $u_\lambda$ be a weak solution of problem \eqref{eq:1.1} obtained by the Mountain Pass Theorem in Lemma~\ref{lemma:2.3}. Then there exists a constant $C>0$ such that for all $0<\lambda<\lambda_3$,
\[
C\,\lambda^{-r}\le \|u_\lambda\|_{L^\infty(\Omega)} .
\]
\end{lemma}

\begin{proof}
Since $E'_\lambda(u_\lambda)=0$, we have
\[
\rho(u_\lambda)^p
=\lambda\int_\Omega f(u_\lambda)u_\lambda\,dx .
\]
Using Lemma~\ref{lemma:2.3} and \eqref{eq:1.4}, we obtain
\[
\begin{aligned}
\lambda\int_\Omega f(u_\lambda)u_\lambda\,dx
&=\rho(u_\lambda)^p \\
&=pE_\lambda(u_\lambda)+p\lambda\int_\Omega F(u_\lambda)\,dx \\
&\ge p c_1\lambda^{-rp}+p\lambda|\Omega|\min F
\ge C_1\,\lambda^{-rp},
\end{aligned}
\]
for some constant $C_1>0$.

On the other hand, by \eqref{eq:1.2} there exists $B>0$ such that
\[
|f(s)s|\le B\big(|s|^{q+1}+|s|\big)
\qquad \text{for all } s\in\mathbb{R}.
\]
Hence
\[
\begin{aligned}
\lambda\int_\Omega f(u_\lambda)u_\lambda\,dx
&\le B\lambda\int_\Omega\big(|u_\lambda|^{q+1}+|u_\lambda|\big)\,dx \\
&\le B\lambda|\Omega|
\big(\|u_\lambda\|_{L^\infty(\Omega)}^{q+1}
+\|u_\lambda\|_{L^\infty(\Omega)}\big).
\end{aligned}
\]

Combining the above estimates yields
\[
C_1\lambda^{-rp}
\le B\lambda|\Omega|
\big(\|u_\lambda\|_{L^\infty(\Omega)}^{q+1}
+\|u_\lambda\|_{L^\infty(\Omega)}\big),
\]
which implies
\[
\|u_\lambda\|_{L^\infty(\Omega)} \ge C\,\lambda^{-r}
\]
for some constant $C>0$ independent of $\lambda$.
\end{proof}

\noindent \textbf{Proof of Theorem \ref{main}.}
We argue by contradiction. Suppose that there exists a sequence 
$\{\lambda_j\}\subset (0,1)$ with $\lambda_j\to0$ and corresponding solutions 
$\{u_{\lambda_j}\}$ such that
\[
m\bigl(\{x\in\Omega: u_{\lambda_j}(x)\le0\}\bigr)>0 .
\]
Define
\[
w_j:=\frac{u_{\lambda_j}}{\|u_{\lambda_j}\|_\infty}.
\]
Then $w_j$ satisfies
\begin{equation}\label{eq:2.21}
-\Delta_p w_j+(-\Delta)^s_p w_j
=\lambda_j g(u_{\lambda_j})
\quad\text{in }\Omega,
\end{equation}
where $g(u_{\lambda_j})=f(u_{\lambda_j})\|u_{\lambda_j}\|_\infty^{1-p}$.

From the previous remark and Lemma~\ref{lemma:2.5}, there exists $C_3>0$
such that
\begin{equation}\label{eq:2.22}
\rho(w_j)\le C_3 .
\end{equation}
Hence $\{w_j\}$ is bounded in $\mathcal{X}_{1,p}(\Omega)$. 
By the regularity result in \citep[Theorem~1.1]{antonini2025global},
$\{w_j\}$ is uniformly bounded in $C^{1,\theta}(\overline{\Omega})$ for some 
$\theta\in(0,1)$. Therefore, up to a subsequence,
\[
w_j\to w \quad\text{in } C^{1,\theta_1}(\overline{\Omega})
\]
for every $\theta_1\in(0,\theta)$. Next, we show that $w\geq 0.$ 
Let $v_0$ be the weak solution of
\begin{equation}\label{eq:2.23}
\begin{cases}
-\Delta_p v_0+(-\Delta)^s_p v_0=1 &\text{in }\Omega,\\
v_0=0 &\text{in }\mathbb{R}^N\setminus\Omega .
\end{cases}
\end{equation}
Define
\[
k_j:=\lambda_j\|u_{\lambda_j}\|_\infty^{1-p}\min\{f(t):t\in\mathbb{R}\}
\]
and let $v_j$ solve
\begin{equation}\label{eq:2.24}
\begin{cases}
-\Delta_p v_j+(-\Delta)^s_p v_j=k_j &\text{in }\Omega,\\
v_j=0 &\text{in }\mathbb{R}^N\setminus\Omega .
\end{cases}
\end{equation}
Then $v_j=(-k_j)^{\frac1{p-1}}v_0$. Since
$\lambda_j g(u_{\lambda_j})\ge k_j$, the comparison principle
(\citep[Proposition~4.1]{antonini2025global}) yields
\[
w_j\ge v_j \quad\text{in }\Omega .
\]
As $v_j\to0$ as $j \rightarrow \infty$, we obtain
\[
w\ge0 \quad\text{in }\Omega .
\]

Let $\ell=\frac{Npr}{N-sp}$. Using \eqref{eq:1.2}, Sobolev embedding,
and Lemma~\ref{lemma:2.5}, one shows that
\[
\{\lambda_j g(u_{\lambda_j})\}\ \text{is bounded in } L^\ell(\Omega).
\]
Hence, up to a subsequence,
\[
\lambda_j g(u_{\lambda_j}) \rightharpoonup z
\quad\text{in }L^\ell(\Omega)
\]
for some $z \in L^{\ell}(\Omega)$. Since $\|u_{\lambda_j}\|_{\infty}^{1-p}\lambda_j \rightarrow 0$ as $j \rightarrow \infty$ and $f$ is bounded from below, $z \geq 0$.

Passing to the limit in \eqref{eq:2.21} using the strong convergence
$w_j\to w$ in $C^{1,\theta_1}(\overline{\Omega})$, we obtain
\[
-\Delta_p w+(-\Delta)^s_p w=z\ge0
\quad\text{in }\Omega .
\]

Thus, $w$ is a weak supersolution of
\[
-\Delta_p w+(-\Delta)^s_p w = 0 \quad\text{in }\Omega,
\qquad w\ge0 \text{ in }\mathbb{R}^N\setminus\Omega .
\]
By the strong maximum principle
(\citep[Proposition~6.1]{antonini2025global}),
\[
w>0 \quad\text{in }\Omega .
\]

Since $w_j\to w$ in $C^{1,\theta_1}(\overline{\Omega})$, we have
$w_j>0$ in $\Omega$ for all sufficiently large $j$, which contradicts
\[
m(\{x\in\Omega:u_{\lambda_j}(x)\le0\})>0 .
\]
Therefore, the claim follows. \qed

\bibliographystyle{abbrv}
\bibliography{ref}
\end{document}